\documentclass[a4paper,12pt]{amsart}
\usepackage{amssymb,amscd,amsmath}

\newcommand{\nexteq}{\displaybreak[0]\\ &=}

\DeclareMathOperator{\rank}{rank}

\newcommand{\Z}{\mathbb{Z}}

\theoremstyle{plain}
\newtheorem{thm}{Theorem}[section]
\newtheorem{prop}[thm]{Proposition}
\newtheorem{lem}[thm]{Lemma}

\theoremstyle{definition}

\theoremstyle{remark}
\newtheorem{rem}{Remark}[section]

\newcommand{\forme}[1]{}
\newcommand{\gn}[1]{\langle {#1}\rangle}

\begin{document}
%============= title ========================%%
\title{Upper bounds on cyclotomic numbers}
\author{Koichi Betsumiya}
\email{betsumi@cc.hirosaki-u.ac.jp}
\address{Graduate School of Science and Technology, Hirosaki University,
Hirosaki 036-8561, Japan}

\author{Mitsugu Hirasaka}\thanks{The second author thanks the support from the grant represented by the third author
when the second author stayed at Hirosaki University from April 22-27 in 2011.}
\email{hirasaka@pusan.ac.kr}
\address{Department of Mathematics, Pusan National University, Jang-jeon dong, Busan 609-735, Korea}

\author{Takao Komatsu}\thanks{The third author is supported in part by the Grant-in-Aid for
Scientific research (C) (No. 22540005), the Japan Society for the
Promotion of Science.}
\email{komatsu@cc.hirosaki-u.ac.jp}
\address{Graduate School of Science and Technology,
Hirosaki University,
Hirosaki 036-8561, Japan}

\author{Akihiro Munemasa}
\email{munemasa@math.is.tohoku.ac.jp}
\address{Graduate School of Information Sciences, Tohoku University,
Sendai 980-8579, Japan}

\date{\today}

%============================================%%

%%%%%%%%%%%%%%%%%%%%%%%%%%%%%%%%%%%%%%%%%%%%%%%%%%%
\begin{abstract}
In this article, we give
%we focus on cyclotomic numbers
% which appears as the intersection numbers of cyclotomic schemes
%containing the number $\lambda$ of common neighbors of two adjacent edge
%in a not-trivial relation. In our main result we give
upper bounds for cyclotomic
%schemes and some criterions on $\lambda$.
numbers of order $e$ over a finite field with $q$ elements,
where $e$ is a positive divisor of $q-1$. In particular, we show
that under certain assumptions, cyclotomic numbers
are at most
%$(q-1)/2e$, and the cyclotomic number $(0,0)$ is at most $(q-1)/2e-1$.
$\lceil\frac{k}{2}\rceil$, and the cyclotomic number $(0,0)$ is at most $\lceil\frac{k}{2}\rceil-1$,
where $k=(q-1)/e$. These results are obtained by using a known formula
for the determinant of a matrix whose entries are binomial coefficients.
\end{abstract}
%%%%%%%%%%%%%%%%%%%%%%%%%%%%%%%%%%%%%%%%%%%%%%%%%%%
\maketitle

%%%%%%%%%%%%%%%%%%%%%%%%%%%%%%%%%%%%%%%%%%%%%%%%%%%
\section{Introduction}\label{sec:intro}

Let $q$ be a power of a prime $p$.
Let $GF(q)$ denote the Galois field with $q$ elements
and let $\alpha$ be a primitive element of $GF(q)$.
According to \cite[Section 2.2]{BE}, for a positive divisor $k$ of $q-1$
we define \textit{cyclotomic numbers} of order $e=\frac{q-1}{k}$ as follows.
For an integer $a$, let $C_a$ denote the cyclotomic coset
$\gn{\alpha^e}\alpha^a$.
For integers $a,b$ with $0\leq a,b< e$, the cyclotomic number $(a,b)$ is
defined as
\[(a,b)=|C_b\cap(C_a+1)|.\]
%|\gn{\alpha^{e}}\alpha^b\cap (\gn{\alpha^e}\alpha^a+1)|.\]
These numbers appear as intersection numbers of a \textit{cyclotomic scheme}
whose non-diagonal relations are Cayley digraphs over $GF(q)$ with
connection set $C_a$
%$\gn{\alpha^e}\alpha^i$
with $0\leq a\leq e-1$ (see \cite[p.~66]{BCN}).
For example, when
$-1\in C_0$,
%$-1\in\langle\alpha^e\rangle$,
these Cayley digraphs
are actually undirected, and
$(0,0)$ is the number of common neighbors of two adjacent vertices.
We remark that all of these Cayley digraphs are pairwise isomorphic,
they are undirected if and only if $k$ is even or $p=2$, and
each of them is a disjoint union of complete graphs
if and only if $k+1$ is a power of $p$, in which case
$(0,0)=k-1$.

Cyclotomic numbers have been studied since the beginning of the last century and
they can
be determined from the knowledge of Gauss sums.
However, explicit evaluation of Gauss sums of large orders is
difficult in general \cite[pp.~98--99 and p.~152]{BE}, so one cannot expect a general formula
for cyclotomic numbers. Instead, we aim to establish upper bounds
for cyclotomic numbers.

In 1972, Wilson \cite{wilson2}
gave an inequality for higher cyclotomic numbers.
This in particular
gives upper and lower bounds for ordinary cyclotomic numbers.
However, the inequality for ordinary cyclotomic numbers is
a consequence of an exact evaluation of
the variance of cyclotomic numbers \cite{wilson}:
%which was improved by himself as follows:
\begin{equation}\label{eq:wilson}
\sum_{a,b=0}^{e-1}((a,b)-\frac{q-2}{e^2})^2=(e-3)k+1+\frac{2k}{e}-\frac{1}{e^2}\leq q-1.
\end{equation}
%The fourth author got to know (\ref{eq:wilson}) though private communication with Wilson
%where there was no published paper to record the equation until now according to Wilson.
For each fixed $e$, we see from (\ref{eq:wilson}) that the cyclotomic
number $(a,b)$ is close to $\frac{k}{e}$, that is,
\begin{equation}\label{O}
(a,b)=\frac{k}{e}+O(\sqrt{k})\quad\text{as }k\to\infty.
\end{equation}
However, when $e\geq k$, the formula does not seem to give
any reasonable bound for $(a,b)$ beyond the trivial bound $(a,b)\leq k$.
This is unavoidable since, when $k+1$ is a power of $p$, $(0,0)=k-1$
as we mentioned earlier.
%Applying (\ref{eq:wilson}) for $(k,e)$ such that $k$ is even with $5\leq e\leq \frac{k}{16}$
%we obtain the following (see Section 2 for the proof):
%\begin{equation}\label{eq:wm}
%(a,b)\leq\frac{k-2}{2}.
%\end{equation}
%We emphasize that (\ref{eq:wilson}) does not give the inequality
%%is very powerful but not almighty, for instance, as in
%(\ref{eq:wm}) in general, as
%some restriction on $e$ and $k$ is needed.
% to obtain the same bound.
%Actually, the case of $(0,0)=k-1$ is out of this restriction and
%(\ref{eq:wilson}) and (\ref{eq:wm}) does not seem to give the same upper bound when $k=e$.
%Following this result we aim to give similar upper bounds
%like a half of $k$
%for cyclotomic numbers.
% and to give some criterion on $(0,0)$ as follows:
The purpose of this paper is to give upper bounds on cyclotomic
numbers without assuming any relations among $e$ and $k$, but
instead, we need to assume that $p$ is sufficiently large compared to $k$.

\begin{thm}\label{thm:main}
Let $q$ be a power of an odd prime $p$ and $k$ a positive divisor of $q-1$.
Then we have the following:
\begin{enumerate}
\item $(a,b)\leq\left\lceil\frac{k}{2}\right\rceil$
for all $a,b$ with $0\leq a,b<e$ if $p>\frac{3k}{2}-1$;
\item $(a,a)\leq \lceil\frac{k}{2}\rceil-1$ for each $a$ with $0\leq a<e$ if $k$ is odd and $p>\frac{3k}{2}$;
\item $(0,0)\leq \lceil\frac{k}{2}\rceil-1$ if $p>\frac{3k}{2}$;
\item $(0,0)=2$ if $p$ is sufficiently large and $6\mid k$;
\item $(0,0)=0$ if $p$ is sufficiently large and $6\nmid k$.
\end{enumerate}
\end{thm}

%Remark that Theorem~\ref{thm:main}(i) and (ii)
%need the assumption on $p$ but not on $e$
% where $p>k+\left\lfloor\frac{k}{2}\right\rfloor-1$
%implies $e\geq 2$, and so it can be applied for the case $e=k$.
%At this point Theorem~\ref{thm:main}(i),(ii)
%can play a complementary role together with (\ref{eq:wilson}) and (\ref{eq:wm}).

%A regular graph is called \textit{edge-regular} if any two adjacent vertices have
%precisely $\lambda$ common neighbors (see \cite[Section 1.1]{BCN}). Thus,
%the symbol $\lambda$ is often used as the number of triangles which contains a given
%edge in an edge-regular graph. Note that each of the above Cayley graphs is edge-regular
%and the $\lambda$ coincides with $(0,0)$.
%Here we remark that $(0,0)$ is an attractive number from a point of view of graph theory.

Note that, if $2k+1$ is a power of $p$, then the upper bounds
in Theorem~\ref{thm:main}(i),(ii) and (iii) are attained.
%Other than the two extremal cases where $k-1$ or $2k-1$ is a power of $p$,
%one may expect that
%$(0,0)$ is much smaller than $k$, but we have not found
%any general upper bounds in the literature.

In the proof of Theorem~\ref{thm:main}(i),(ii) and (iii)
we use a formula to expand the determinant of the matrix
\begin{equation}\label{eq:bi}
\left(\binom{r+s}{r-i+j}\right)_{1\leq i,j\leq m}
\end{equation}
given in \cite{ADC}. In Section~3 we show that the cyclotomic number $(a,b)$
is equal to $k-\rank C^{(a,b)}$, where $C^{(a,b)}$ is a certain matrix
with entries in $GF(q)$ (see Lemma~\ref{lem:1}).
Thus, giving a lower bound for the rank of $C^{(a,b)}$ results in
an upper bound for the cyclotomic number $(a,b)$.
Since the matrix $C^{(a,b)}$ contains a submatrix which is the modulo $p$
reduction of (\ref{eq:bi}) for suitable $r$ and $s$,
we obtain a lower bound for the rank whenever the determinant does not
vanish modulo $p$.
%Thus, we apply a known formula on the determinant of an integral matrix for a matrix over a finite field to obtain
%upper bounds for the cyclotomic numbers.

Though we reached \cite{ADC} via \cite{PW},
there is a typo in the formula (2) in \cite{PW}, so that
the simple expression (5) in \cite{PW} does not give the
evaluation of the above determinant.

In Section~2 we will establish Wilson's formula (\ref{eq:wilson}).
% the average and variance of cyclotomic numbers according to \cite{wilson},
We include its proof as it has not been published.
In Section~3 we will prepare some results to prove Theorem~\ref{thm:main} in Section~4.
In Section~5, we show that the inequality
$(a,b)\leq\left\lceil\frac{k}{2}\right\rceil$ holds under some
assumptions different from the one in Theorem~\ref{thm:main}(i).

\section{Wilson's formula}
For the remainder of this article  we use the same notation as in Section 1
and we shall write $GF(q)$ as $F$ for short.

%Now we will show some equations given by Wilson.

\begin{lem}[R. M. Wilson]\label{lem:Wilson}
Let
\begin{align*}
X&=\{(x,y)\in(F\setminus\{0,1\})^2\mid x\neq y,\;x\in yC_0,\;x-1
\in (y-1)C_0\},\\
Y&=\{(u,v)\in(C_0\setminus\{1\})^2\mid u\neq v\}.
\end{align*}
Then $f:X\to Y$ defined by
\[
f(x,y)=\left(\frac{x}{y},\frac{x-1}{y-1}\right)
\]
is a bijection.
In particular, $|X|=(k-1)(k-2)$.
\end{lem}
\begin{proof}
If $(x,y)\in X$, then $\frac{x}{y},\frac{x-1}{y-1}\in C_0\setminus\{1\}$.
Since $x\neq y$, we have $\frac{x}{y}\neq\frac{x-1}{y-1}$. Thus
$f(x,y)\in Y$. The inverse mapping of $f$ is given by
\[
g(u,v)=\left(\frac{u(1-v)}{u-v},\frac{1-v}{u-v}\right).
\]
Since $|Y|=(k-1)(k-2)$, the second statement follows.
\end{proof}

\begin{lem}\label{lem:sum1}
We have the following:
\begin{enumerate}
\item $\sum_{a,b=0}^{e-1}(a,b)=q-2$;
\item $\sum_{a,b=0}^{e-1}(a,b)^2
=(k-1)(k-2)+q-2$.
\end{enumerate}
\end{lem}
\begin{proof}
(i)
\begin{align*}
\sum_{a,b=0}^{e-1}(a,b)
&=\sum_{a,b=0}^{e-1}|(C_a+1)\cap C_b|
\nexteq
\sum_{a=0}^{e-1}|(C_a+1)\cap F^\times|
\nexteq
|(F^\times+1)\cap F^\times|
\nexteq
|F\setminus\{0,1\}|
\nexteq
q-2.
\end{align*}

(ii) \begin{align*}
&\sum_{a,b=0}^{e-1}(a,b)^2-\sum_{a,b=0}^{e-1}(a,b)
\\&=\sum_{a,b=0}^{e-1}|\{(x,y)\in((C_a+1)\cap C_b)^2\mid x\neq y\}|
\nexteq
|\bigcup_{a,b=0}^{e-1}\{(x,y)\in((C_a+1)\cap C_b)^2\mid x\neq y\}|
\nexteq
|\{(x,y)\in(F\setminus\{0,1\})^2\mid x\neq y,\\
&\qquad
\exists a\in\{0,1,\dots,e-1\},\;\{x-1,y-1\}\subset C_a,\\
&\qquad
\exists b\in\{0,1,\dots,e-1\},\;\{x,y\}\subset C_b\}|
\nexteq
|\{(x,y)\in(F\setminus\{0,1\})^2\mid x\neq y,\;
x\in yC_0,\;x-1\in (y-1)C_0\}|
\nexteq
|\{(u,v)\in C_0\setminus\{1\}\mid u\neq v\}|
\nexteq
(k-1)(k-2)
\end{align*}
by Lemma~\ref{lem:Wilson}.
\end{proof}

%\begin{flushleft}
%Proof of (\ref{eq:wilson}) and (\ref{eq:wm})
%\end{flushleft}

As mentioned in Section~1, Wilson showed the variance of cyclotomic numbers
through \cite{wilson}.
\begin{thm}[Wilson]\label{thm:wilson}
\[
\sum_{a,b=0}^{e-1}((a,b)-\frac{q-2}{e^2})^2
=(e-3)k+\frac{2k}{e}+1-\frac{1}{e^2}.
\]
\end{thm}
\begin{proof}
By Lemma~\ref{lem:sum1}, we have
\begin{align*}
\sum_{a,b=0}^{e-1}((a,b)-\frac{q-2}{e^2})^2
&=
\sum_{a,b=0}^{e-1}(a,b)^2-\frac{2(q-2)}{e^2}
\sum_{a,b=0}^{e-1}(a,b)+\frac{(q-2)^2}{e^2}
\nexteq
(k-1)(k-2)+q-2-\frac{2(q-2)^2}{e^2}+\frac{(q-2)^2}{e^2}
%\nexteq
%(k-1)(k-2)+ek-1-\frac{(ek-1)^2}{e^2}
%\nexteq
%k^2-3k+2+ek-1-k^2+\frac{2k}{e}-\frac{1}{e^2}
\nexteq
(e-3)k+\frac{2k}{e}+1-\frac{1}{e^2}.
\end{align*}
\end{proof}

%\\&\leq
%ek-2k+1
%\\&\leq ek=q-1.

Since
\[
(e-3)k+\frac{2k}{e}+1-\frac{1}{e^2}\leq ek,
\]
Theorem~\ref{thm:wilson} implies
\begin{equation}\label{r1}
|(a,b)-\frac{q-2}{e^2}|<\sqrt{ek},
\end{equation}
Thus, for $e$ fixed and $k\to\infty$,
\begin{align*}
\frac{1}{\sqrt{k}}|(a,b)-\frac{k}{e}|
&\leq
\frac{1}{\sqrt{k}}(|(a,b)-\frac{q-2}{e^2}|+\frac{1}{e^2})
\\&\leq
\sqrt{e}+\frac{1}{e^2\sqrt{k}}
\end{align*}
which is bounded. This implies (\ref{O}).

\begin{rem}
By (\ref{r1}), we have
\begin{align*}
(a,b)&<
%\frac{ek-1}{e^2}+\sqrt{ek} \\&<
\frac{k}{e}+\sqrt{ek}.
\end{align*}
If $\frac{k}{16}\geq e\geq4$, then
\begin{align*}
(a,b)&<
\frac{k}{4}+\sqrt{\frac{k^2}{16}}
\nexteq
%\frac{k}{5}+\frac{k}{4}
%\\&<
\frac{k}{2}.
\end{align*}
%Thus, for $k$ even,
%\[
%(a,b)\leq\frac{k-2}{2}.
%\]
%\end{proof}
We shall show in the next section that a similar inequality holds
without any assumption on $e$ and $k$, if $p$ is sufficiently large.
\end{rem}

\section{Cyclotomic numbers and determinants}

Define a $k\times k$ matrix $C^{(a,b)}$ with entries in $F$ for integers $a$, $b$
by
\[
\left(C^{(a,b)}\right)_{i,j}=
\begin{cases}
1+\alpha^{ak}-\alpha^{bk}&\text{if $i=j$,}\\
\binom{k}{j-i}&\text{if $i<j$,}\\
\alpha^{ak}\binom{k}{i-j}&\text{otherwise.}
\end{cases}
\]

\begin{lem}\label{lem:1}
\[
(a,b)=k-\rank C^{(a,b)}=\deg(\gcd((X+1)^k-\alpha^{bk},X^k-\alpha^{ak})).
\]

\end{lem}
\begin{proof}
For simplicity, write
$\beta=\alpha^{ak}$ and
$\gamma=1+\beta-\alpha^{bk}$.
Let $T$ denote the following $k\times k$ matrix:
\[
T=\begin{pmatrix}
&1&&\\
&&\ddots&\\
&&&1\\
\beta&&&
\end{pmatrix}.
\]
Then $T$ has the characteristic polynomial
$\psi(X)=X^k-\beta$ which has $k$ distinct roots in $F$.
Define another polynomial $\phi(X)$ by
\[
\phi(X)=(X+1)^k-\alpha^{bk}.
\]
Then
\begin{align*}
C^{(a,b)}&=\gamma I+\sum_{i=1}^{k-1}\binom{k}{i}T^i
\nexteq
(1+\beta-\alpha^{bk})I+\sum_{i=1}^{k-1}\binom{k}{i}T^i
\nexteq
I+T^k-\alpha^{bk}I+\sum_{i=1}^{k-1}\binom{k}{i}T^i
%\nexteq
%\sum_{i=0}^{k}\binom{k}{i}T^i-\alpha^{bk}I
\nexteq
(T+I)^k-\alpha^{bk}I
\nexteq
\phi(T).
\end{align*}
The multiplicity of $0$ as an eigenvalue of $C^{(a,b)}$ is
the number of eigenvalues $\theta$ of $T$ with
$\phi(\theta)=0$. Thus
\begin{align*}
k-\rank C^{(a,b)}&=
|\{\theta\mid\theta\text{ is an eigenvalue of $T$, }
\phi(\theta)=0\}|
\nexteq
|\{x\in F\mid \psi(x)=\phi(x)=0\}|
\nexteq
|\{x\in F\mid x^k=\alpha^{ak},\;(x+1)^k=\alpha^{bk}\}|
\nexteq
|\{x\in F\mid x\in\langle\alpha^e\rangle\alpha^a,\;x+1
\in \langle\alpha^e\rangle\alpha^b\}|
\nexteq
%|(\langle \alpha^e\rangle\alpha^a+1)\cap
%\langle \alpha^e\rangle\alpha^b|
|(C_a+1)\cap C_b|
\nexteq
(a,b).
\end{align*}
Since $(a,b)$ is the number of common roots in $F$ of $\phi(X)$ and $\psi(X)$,
the second equality holds.
\end{proof}

Let $m=\left\lfloor\frac{k}{2}\right\rfloor$.
The upper right $m\times m$ submatrix of $C^{(a,b)}$ is
\begin{equation}\label{ur}
\begin{pmatrix}
\binom{k}{m}&\binom{k}{m+1}&\cdots&\binom{k}{2m-1}\\
\vdots&\vdots&&\vdots\\
\binom{k}{1}&\binom{k}{2}&\cdots&\binom{k}{m}
\end{pmatrix}
\:\:\mbox{if $k=2m$},
\end{equation}

\begin{equation}\label{urodd}
\begin{pmatrix}
\binom{k}{m+1}&\binom{k}{m+2}&\cdots&\binom{k}{2m}\\
\vdots&\vdots&&\vdots\\
\binom{k}{2}&\binom{k}{3}&\cdots&\binom{k}{m+1}
\end{pmatrix}
\:\:\mbox{if $k=2m+1$},
\end{equation}
whose determinants can be calculated.

\begin{lem}[{\cite[Section 2.2]{ADC}}]\label{lem:ADC}
Let $r,s,m$ be integers with $r,s,m\geq0$. Then
\[
\det\left(\binom{r+s}{r-i+j}\right)_{1\leq i,j\leq m}
=\prod_{i=0}^{m-1}
\frac{i!(r+s+i)!}{(r+i)!(s+i)!}.
\]
\end{lem}
By Lemma~\ref{lem:ADC}, the matrices (\ref{ur}) and (\ref{urodd}) have determinants
\begin{equation}\label{urd}
\prod_{i=0}^{m-1}
\frac{i!(k+i)!}{(m+i)!(k-m+i)!}\:\:\mbox{if $k=2m$},
\end{equation}
\begin{equation}\label{urdodd}
\prod_{i=0}^{m-1}
\frac{i!(k+i)!}{(m+1+i)!(k-m-1+i)!}\:\:\mbox{if $k=2m+1$},
\end{equation}
respectively.

Similarly, the matrix
\begin{equation}\label{ur1}
\begin{pmatrix}
\binom{k}{m}&\binom{k}{m+1}&\cdots&\binom{k}{2m-1}&\binom{k}{2m}\\
\vdots&\vdots&&\vdots\\
\binom{k}{1}&\binom{k}{2}&\cdots&\binom{k}{m}&\binom{k}{m+1}\\
\binom{k}{0}&\binom{k}{1}&\cdots&\binom{k}{m-1}&\binom{k}{m}
\end{pmatrix}
\end{equation}
has determinant
\begin{equation}\label{urd1}
\prod_{i=0}^{m}
\frac{i!(k+i)!}{(m+i)!(k-m+i)!}.
\end{equation}

\section{Bounds}

\begin{prop}\label{prop:1}
If
%$p>\left\lfloor\frac{3k}{2}\right\rfloor-1$,
$p>\frac{3k}{2}-1$,
then
%the cyclotomic number $(a,b)$ of order $(q-1)/k$ over $F$
%for any prime power $q$ of $p$, is at most
$(a,b)\leq\left\lceil\frac{k}{2}\right\rceil$
for all $a,b$ with $0\leq a,b<e$.
\end{prop}
\begin{proof}
Let $m=\left\lfloor\frac{k}{2}\right\rfloor$.
%If $p>k+m-1$ (or equivalently $p>\frac{3k}{2}-1$), then
%If $p>\left\lfloor\frac{3k}{2}\right\rfloor-1$, then
Since $p>k+m-1$,
the determinant (\ref{urd}) (resp.\ (\ref{urdodd}))
is nonzero modulo $p$ when $k=2m$ (resp.\ $k=2m+1$), hence
$\rank C^{(a,b)}\geq\left\lfloor\frac{k}{2}\right\rfloor$.
The result then follows from Lemma~\ref{lem:1}.
\end{proof}

The result of Proposition~\ref{prop:1} can be improved for
$\lambda=(0,0)$.

\begin{prop}\label{prop:0}
Suppose
%$p>\left\lfloor\frac{3k}{2}\right\rfloor$,
$p>\frac{3k}{2}$.
Then we have the following:
\begin{enumerate}
\item If $k$ is odd, then $(a,a)\leq \left\lceil\frac{k}{2}\right\rceil-1$ for all $a$ with $0\leq a<e$;
\item $(0,0)\leq \left\lceil\frac{k}{2}\right\rceil-1$.
\end{enumerate}
\end{prop}
\begin{proof}
If $k=2m+1$, then the upper right
$(m+1)\times(m+1)$ submatrix of $C^{(a,a)}$ is given by (\ref{ur1}).
If $k=2m$, then the
$(m+1)\times(m+1)$ submatrix of $C^{(0,0)}$ consisting
of the first $m+1$ rows, the upper right $m$ columns and
the first column is also given by (\ref{ur1}).
Since $p>\frac{3k}{2}$,
the determinant (\ref{urd1}) is nonzero modulo $p$, hence
$\rank C^{(a,a)}\geq \frac{k+1}{2}$ if $k$ is odd,
and $\rank C^{(0,0)}\geq \frac{k}{2}+1$ if $k$ is even.
The result then follows from Lemma~\ref{lem:1}.
\end{proof}

%If
%%$p=\characteristic F$ and
%$\det C^{(0,0)}\not\equiv0\pmod{p}$, then
%we see that $\lambda=(0,0)=0$ by Lemma~\ref{lem:1}.
%This means that we need to
%compute $\lambda$ only for those primes $p$ which divide
%$\det C^{(0,0)}\not\equiv0\pmod{p}$.
%We do not know any formula for $\det C^{(0,0)}$.
%it is usually nonzero over $\Z$.

\begin{prop}\label{prop:0p}
Let $\lambda$ denote the cyclotomic number $(0,0)$.
% of order $e=(q-1)/k$ over $F$,
%where $q$ is a power of a prime $p$.
If $p$ is sufficiently large,
then $\lambda=2$ or $\lambda=0$, according as
$k\equiv0\pmod6$ or not.
\end{prop}
\begin{proof}
Since the matrix $C^{(0,0)}$ does not involve $\beta$ or $\gamma$
as entries, we may regard it as a matrix over $\Z$.
The eigenvalues of $C^{(0,0)}$ are
\[
(\zeta^j+1)^k-1\quad(j=0,1,\dots,k-1)
\]
where $\zeta=\exp\frac{2\pi i}{k}$. This is zero only if
$k\equiv0\pmod6$, and
$\zeta^j=\exp\frac{2\pi i}{3}$ or $\exp\frac{4\pi i}{3}$.
If $k\not\equiv0\pmod6$, all the eigenvalues of
$C^{(0,0)}$ are nonzero, hence $C^{(0,0)}$ is invertible
in characteristic $0$. This implies that $\det C^{(0,0)}$
is a nonzero integer. If $p>|\det C^{(0,0)}|$,
then $C^{(0,0)}$ is invertible in
characteristic $p$, hence $\lambda=0$ by Lemma~\ref{lem:1}.

If $k\equiv0\pmod6$, then the multiplicity of $0$ as an
eigenvalue of $C^{(0,0)}$ is $2$. This implies that
$C^{(0,0)}$ contains a $(k-2)\times(k-2)$ minor whose determinant
is nonzero in characteristic $0$. Thus, if $p$ is sufficiently large,
$C^{(0,0)}$ has rank $k-2$,
hence $\lambda=2$ by Lemma~\ref{lem:1}.
\end{proof}

When $k\not\equiv0\pmod6$, there are only finitely many primes $p$
which divide $\det C^{(0,0)}$. This means that the set of
characteristics $p$ for which $\lambda>0$ holds is finite.
However, as we do not know any formula for $\det C^{(0,0)}$,
we do not know the set of characteristics for which $\lambda>0$ holds.

%Proof of Theorem~\ref{thm:main}.
\begin{proof}[Proof of Theorem~\ref{thm:main}]
%Since $(0,0)$,
All the statements are direct consequences of
the above propositions.
\end{proof}

\section{Additional results}

%For $f_1,f_2,\ldots, f_t\in F[x]$, we shall write the ideal of $F[X]$
%generated by $f_1, f_2,\ldots, f_t$
%by $\gn{f_1,f_2,\ldots, f_t}$.
We set
$\beta=\alpha^{ak}$, $\gamma=1+\beta-\alpha^{bk}$, and define
$\phi,\psi\in F[X]$ by
\[\phi(X)=(X+1)^k-\alpha^{bk},\:
%\[\phi(X)=(X+1)^k-\alpha^{bk},\:\phi_0(X)=\gamma+\sum_{i=1}^{k-1}\binom{k}{i}X^i,\:
\psi(X)=X^k-\beta,\]
%where $\beta=\alpha^{ak}$ and $\gamma=1+\beta-\alpha^{bk}$ are
as in the proof of Lemma~\ref{lem:1}.
Define $\phi_0\in F[X]$ by
\begin{align*}
\phi_0(X)&=\phi(X)-\psi(X)
\\&=\gamma+\sum_{i=1}^{k-1}\binom{k}{i}X^i.
\end{align*}
%Since $\phi(X)=\phi_0(X)+\psi(X)$,
%\[\gn{\phi,\psi}=\gn{\phi_0,\psi}.\]
Let $J$ denote the ideal of $F[X]$ generated by $\phi_0$ and $\psi$.
Then it follows from Lemma~\ref{lem:1} that
%\[(a,b)=\deg(\gcd(\phi_0,\psi)).\]
\begin{equation}\label{J}
(a,b)=\min\{\deg\rho\mid0\neq\rho\in J\}.
\end{equation}

\begin{prop}
If $\frac{3k}{4}\leq p<k$, then $(a,b)\leq  \left\lfloor\frac{k}{2}\right\rfloor$ for all $a,b$ with $0\leq a,b<e$.
\end{prop}
\begin{proof}
We claim that, for each $i$ with $k-p+1\leq i\leq p-1$,
\[\binom{k}{i}\equiv 0\pmod{p}.\]
This follows since $k(k-1)\cdots (k-i+1)\equiv0\pmod{p}$
while $i!\not\equiv0\pmod{p}$.
%The denominator of
%\[\binom{k}{i}=\frac{k(k-1)\cdots (k-i+1)}{i(i-1)\cdots 1}\]
%is nonzero modulo $p$ for each $i$ with $0\leq i\leq p-1$.
%Since $p$ appears in
%\[k(k-1)\cdots (k-i+1)\]
%if $k\geq p\geq k-i+1$, the claim follows.

Note that
\begin{align*}
&X^{k-p}\phi_0(X)-\sum_{i=p}^{k-1}\binom{k}{i}\psi(X)X^{i-p}
\\&=\gamma X^{k-p}+\sum_{i=1}^{k-p}\binom{k}{i}X^{i+k-p}
+\beta\sum_{i=p}^{k-1}\binom{k}{i}X^{i-p}\\
&\quad+\sum_{i=k-p+1}^{p-1}\binom{k}{i}X^{i+k-p},
\end{align*}
and the last summand is zero by the claim.
%\begin{align*}
%X^{k-p}\phi_0(X)
%&=\gamma X^{k-p}+\sum_{i=1}^{k-p}\binom{k}{i}X^{i+k-p}\\
%&\quad+\sum_{i=k-p+1}^{p-1}\binom{k}{i}X^{i+k-p}\\
%&\quad+(\beta+\psi(X))\sum_{i=0}^{k-p-1}\binom{k}{p+i}X^i.
%\end{align*}
%\begin{align*}
%\gamma X^{k-p}+\binom{k}{1}X^{k-p+1}
%+\cdots+
%\binom{k}{k-p}X^{2k-2p}+\\
%\binom{k}{k-p+1}X^{2k-2p+1}+\cdots+\binom{k}{p-1}X^{k-1}+\\
%(\beta+\psi(X))\binom{k}{p}+(\beta+\psi(X))\binom{k}{p+1}X+\cdots +(\beta+\psi(X))\binom{k}{k-1}X^{k-p-1}.
%\end{align*}
%By the claim, the second line of the above equation is zero.
Since
\[\binom{k}{k-p}=\binom{k}{p}\not\equiv 0\pmod{p}\]
by the assumption,
$J$ contains
%$\gn{\phi_0,\psi}$ has
a nonzero polynomial of degree $2k-2p$.
It follows from
%Lemma~\ref{lem:1}
(\ref{J})
and the assumption on $p$ that
\[(a,b)\leq 2k-2p\leq 2k-2\cdot \frac{3k}{4}=\frac{k}{2}.\]
This implies that $(a,b)\leq \left\lfloor\frac{k}{2}\right\rfloor$.
\end{proof}

\begin{prop}
If $k+1<p^t<\frac{3k}{2}$ for some positive integer $t$, then $(a,b)\leq \left\lfloor\frac{k}{2}\right\rfloor$
and $(a,a)\leq \left\lfloor\frac{k}{2}\right\rfloor-1$ for all $a,b$ with $0\leq a,b<e$.
\end{prop}
\begin{proof}
We set $m=\left\lfloor\frac{k}{2}\right\rfloor$.
For each positive integer $i$ with $i\leq k-1$,
we define $\phi_i\in J$ by
%\begin{equation}\label{eq:67}
\[\phi_i(X)=X^i\phi_0(X)-\psi(X)\sum_{j=0}^{i-1}\binom{k}{i-j}X^j.\]
%\end{equation}
%By the definition of $\phi_0$ and $\psi$, $\phi_i(X)\in \gn{\phi_0,\psi}$
Then
\[
\phi_i(X)=\gamma X^i+\beta\sum_{j=0}^{i-1}\binom{k}{i-j}X^j
+\sum_{j=i+1}^{k-1}\binom{k}{j-i}X^j,
\]
and hence $\deg(\phi_i)\leq k-1$.
%for each $i$ with $0\leq i\leq k-1$.
Let
\[f(X)=\sum_{i=0}^{m}\binom{p^t-k}{m-i}\phi_{i}(X)
=\sum_{l=0}^{k-1}a_lX^l
.\]
%We denote the coefficient of $X^l$ in $f(X)$ and
%$\sum_{i=0}^{m}\binom{p^t-k}{i}X^{m-i}\phi_0(X)$
%by $a_l$ and $b_l$, respectively.
For $0\leq i\leq l<k$, the coefficient of $X^l$ in
\[
\psi(X)\sum_{j=0}^{i-1}\binom{k}{i-j}X^j
\]
is zero. This implies that $\phi_i(X)$ and $X^i\phi_0(X)$
have the same coefficient at degree $l$.
%We claim that $a_l=0$ for each $l$ with $m<l \leq k-1$.
%Note that $a_l=b_l$ if $m<l \leq k-1$.
Thus, for $m< l\leq k-1$,
\begin{align*}
%a_l=b_l
a_l
&=\sum_{i=0}^m\binom{p^t-k}{m-i}\binom{k}{l-i}
\nexteq
%\sum_{i=0}^m\binom{p^t-k}{m-i}\binom{k}{k-l+i}\\
%&=\sum_{m-i+j=l}\binom{p^t-k}{i}\binom{k}{j}\\
%&=\sum_{m-i+j=l}\binom{p^t-k}{i}\binom{k}{k-j}\\
%&=\sum_{i+j'=m-l+k}\binom{p^t-k}{i}\binom{k}{j'}\\
%&=
\binom{p^t}{m+k-l}
%\equiv 0\pmod{p},
\nexteq0,
\end{align*}
%if $m<l\leq k-1$, since $p^t-k\leq m$ by the assumption.
and
\begin{align*}
a_m&=
\sum_{i=0}^{m-1}\binom{p^t-k}{m-i}\binom{k}{m-i}+\gamma
\nexteq
\binom{p^t}{k}+\gamma-1
\nexteq
\gamma-1.
\end{align*}
%We claim that $a_m=\gamma-1$.
%Note that $a_m=b_m$ and
%\begin{align*}
%b_m
%&=\gamma+\sum_{i=1}^{m}\binom{p^t-k}{i}\binom{k}{i}\\
%&=\gamma-1+\sum_{i=0}^{m}\binom{p^t-k}{i}\binom{k}{i}\\
%&=\gamma-1+\binom{p^t}{k}=\gamma-1.
%\end{align*}

We claim that $a_{m-1}=\beta k$ if $\gamma=1$.
Since the coefficient of $X^{m-1}$ in
\[
\psi(X)\sum_{j=0}^{m-1}\binom{k}{m-j}X^j
\]
is $-\beta k$, we have
%Note that $a_{m-1}-b_{m-1}=\beta\binom{k}{k-1}$
%and
\begin{align*}
a_{m-1}
&=\sum_{i=0}^{m-1}\binom{p^t-k}{m-i}\binom{k}{m-1-i}+\beta k
\nexteq
\sum_{j=0}^{m-1}\binom{p^t-k}{j+1}\binom{k}{j}+\beta k
\nexteq
\sum_{j=0}^{k}\binom{p^t-k}{j+1}\binom{k}{k-j}+\beta k
&&\text{(since $p^t-k\leq m$)}
\nexteq
\binom{p^t}{k+1}+\beta k
\nexteq
\beta k.
%&=\binom{p^t-k}{1}+\sum_{i=2}^{m}\binom{p^t-k}{i}\binom{k}{i-1}\\
%&=\sum_{i=1}^{m}\binom{p^t-k}{i}\binom{k}{i-1}\\
%&=\sum_{j=0}^{m-1}\binom{p^t-k}{j+1}\binom{k}{j}\\
%&=\binom{p^t}{k+1}\equiv 0\pmod{p}.
\end{align*}
%Thus, the claim holds.
Therefore, $\deg(f)=m$ if $\gamma\ne 1$, and $\deg(f)=m-1$ if $\gamma=1$ or equivalently $a=b$.
Since $f\in J$, it follows from (\ref{J}) that $(a,b)\leq m$ and $(a,a)\leq m-1$.
\end{proof}

\noindent {\bf Acknowledgements.}
The authors would like to thank Richard M. Wilson for suggesting
the use of his formula for the variance of cyclotomic numbers.

%%%%%%%%%%%%%%%%%%%%%%%%%%%%%%%%%%%%%%%%%%%%%%%%%%%%%%%%%%%%%%%%%%%%%%%%%%%%%%%%%%%%%%%%%%%%%%%%

\end{document}